\documentclass[10pt]{article}
\usepackage{amsmath,amssymb,amsfonts,amsthm,mathtools,epsfig}
\usepackage{tikz}
\usepackage{caption,subcaption,verbatim}
\usepackage{setspace}
\usepackage{algorithm}
\usepackage{algorithmic}
\usepackage{xurl}
\usepackage{hyperref}
\usepackage{ulem}

\usepackage{cancel}

\usepackage[bottom=1.5in]{geometry}

\newtheorem{thm}{Theorem}
\newtheorem{theorem}[thm]{Theorem}
\newtheorem{lemma}[thm]{Lemma}
\newtheorem{corollary}[thm]{Corollary}

\newtheorem{prop}[thm]{Proposition}

\theoremstyle{definition}

\date{ }

\title{\bf On $(k,g)$-Graphs without $(g+1)$-Cycles}

\author{
Leonard Chidiebere Eze \thanks{Department of Algebra and Geometry, Faculty of Mathematics, Physics and Informatics, Mlynsk\' a dolina, 842 48 Bratislava, Slovakia. 
}
\and Robert Jajcay
 \footnotemark[1]
\and Jorik Jooken
 \thanks{Department of Computer Science, KU Leuven Kulak, 8500 Kortrijk, Belgium.\\ E-mail addresses:
 \protect\href{mailto:leonard.eze@fmph.uniba.sk}{\protect\nolinkurl{leonard.eze@fmph.uniba.sk}},
\protect\href{mailto:robert.jajcay@fmph.uniba.sk}{\protect\nolinkurl{robert.jajcay@fmph.uniba.sk}} and \protect\href{mailto:jorik.jooken@kuleuven.be}{\protect\nolinkurl{jorik.jooken@kuleuven.be}}
}
}

\date{}

\begin{document}
	
	\maketitle

	\begin{abstract}
	 A $(k,g,\underline{g+1})$-graph is a $k$-regular graph of girth $g$ which does not contain cycles of length
     $g+1$. Such graphs are known to exist for all parameter pairs $k \geq 3, g \geq 3 $, and we focus on determining the orders $n(k,g,\underline{g+1})$ of the smallest $(k,g,\underline{g+1})$-graphs. This problem can be viewed
     as a special case of the previously studied \textit{Girth Pair Problem}, the problem of finding the order of a 
     smallest $k$-regular graph in which the length of a smallest even length cycle and the length of a smallest
     odd length cycle are prescribed. When considering the case of an odd girth $g$, this problem also yields results towards the \textit{Cage Problem}, the problem of finding the order of a smallest $k$-regular graph of girth $g$.  
     We establish the monotonicity of the function $n(k,g,\underline{g+1})$ with respect to increasing $g$, and
     present universal lower bounds for the values $n(k,g,\underline{g+1})$.
     We propose an algorithm for generating all $(k,g,\underline{g+1})$-graphs on $n$ vertices, use this algorithm to determine several of the smaller values $n(k,g,\underline{g+1})$, and discuss various approaches to 
     finding smallest $(k,g,\underline{g+1})$-graphs within several classes of highly symmetrical graphs.
			\vskip 3mm
		
			\end{abstract}
   
	\section{Motivation and Background}\label{Intro}
	This study focuses on $k$-regular graphs of girth $g$ that prohibit $(g+1)$-cycles, with emphasis on the case when the girth $g$ is odd. 
 This problem can be viewed as a special case of the \textit{Girth Pair Problem} \cite{Babuena, Babuena2, Harary}. For given $ k \geq 3 $ and a pair of numbers $ 3 \leq g < h $ of
 different parity, the Girth Pair problem calls for finding a smallest $k$-regular graph of girth $g$ in which the length of a shortest cycle of parity opposite to $g$ is $h$.
 It is, however, important to note that unlike the Girth Pair Problem, we are searching for $k$-regular graphs of girth $g$ which contain no cycles of length
 $g+1$, but we do not insist that the length of a shortest cycle of parity opposite to $g$ is $g+3$ (or any other specific number greater than $g+3$). One of the most general results concerning the possible distributions of cycle
 lengths at the beginning of the cycle spectrum of a $k$-regular graph has been proven in~\cite{Pisanski}. That paper proved that for any increasing sequence $g_1, g_2, \dots, g_s$ and any $k \geq 3$, there exist $k$-regular graphs with cycles of length not exceeding $g_s$ existing exactly for the specified lengths $ g_1,g_2,\ldots,g_s $ and containing no cycles of other lengths not exceeding $g_s$. On the 
 other hand, for odd girths $g$, all the {\it extremal} graphs we have found and included in this article contain an even length cycle of length $g+3$. It appears likely that 
a smallest $k$-regular graph of odd girth $g$ that does not contain a $(g+1)$-cycle necessarily contains a $(g+3)$-cycle. This observation is in some sense in line with another result overwhelmingly supported by empirical evidence,
namely the fact that all known extremal graphs of odd girth $g$ also contain a $(g+1)$-cycle. 

 The initial motivation for our study stems from the well-known \textit{Cage Problem}. A $k$-regular graph of girth $g$ is referred to as a $(k,g)$-\textit{graph}. The Cage Problem involves finding a smallest $(k,g)$-graph for a given pair $(k,g)$, $ k \geq 3, g \geq 3$; called a $(k,g)$-\textit{cage}. The Cage Problem is challenging because there is no unique method for constructing cages. Most cages are found using a combination of methods such as excision, regular lifts of graphs, removal of vertices and
 edges from graphs obtained from generalized $n$-gons, various algebra based methods, and extensive computer searches \cite{ExooR1}. 
In particular, the \textit{canonical double cover construction} is a simple lift technique used in this area which starts from a base graph $\Gamma$ of order $n$, replaces
each vertex $u$ of $\Gamma$ by a pair of vertices $ u_1$ and $u_2$, and replaces each edge $uv$ by the pair of edges $u_1v_2$ and $ u_2v_1$.
Despite its simplicity, this construction has been instrumental in establishing a fundamental recursive upper bound. Namely, it is easy to see that the
canonical double cover $\Gamma'$ of a $k$-regular graph $\Gamma$ is bipartite, $k$-regular, and that odd length cycles in the base graph $\Gamma$ lift into even length cycles of twice the lengths in $\Gamma'$ (while the even length cycles lift into cycles of unchanged
length). $\Gamma'$ contains no odd cycles because it is bipartite.  Moreover, every even length cycle in $\Gamma'$ corresponds to a closed walk in $\Gamma$ that does not consecutively visit $u$, $v$ and then $u$ again for any $u,v \in V(\Gamma)$. Hence, the canonical double 
cover of a $(k,g)$-graph $\Gamma$ of \textit{odd} girth $g$ is a $(k,g')$-graph $\Gamma'$ of order twice as big as the order of $\Gamma$. The girth $g'$ is greater than $g$ and equal to the smaller of the length of the shortest even length cycle in $\Gamma$ and of $2g$. This means, in particular, that the canonical double cover of a $k$-regular graph of odd girth $g$ is a $k$-regular graph of girth $g' \geq g+1 $. 
While this observation has been repeatedly exploited \cite{Babuena}, in what follows, we focus on a more specific use of the canonical double cover.

Specifically, the canonical double cover of a $k$-regular graph $\Gamma$ of \textit{odd girth} $g$ \textit{that does not include a $(g+1)$-cycle} is a $k$-regular graph $\Gamma'$ of girth $g' \geq g+3 $ (with  $g' = g+3 $ if and only if $g=3$ or $\Gamma$ contains a $(g+3)$-cycle). This simple, more specialized, observation yields a recursive construction of a $(k,g')$-graph from a $k$-regular graph $\Gamma$ of odd girth $g$ that does not include a $(g+1)$-cycle with $ g' \geq g+3$. Obviously, the smaller the order of $\Gamma $, the smaller the order of the canonical double cover. Thus, in order to construct small $(k,g)$-graphs
in this way, it is natural to search for smallest $k$-regular graphs of odd prescribed girth $g$ that do not contain $(g+1)$-cycles. As mentioned already in the beginning, $k$-regular graphs of (any) prescribed girth $g$ that do not contain $(g+1)$-cycles exist for all pairs $ k,g \geq 3 $ \cite{Pisanski}. Hence, it makes sense to look for the smallest among all such graphs. We shall denote \textit{the order of a smallest $k$-regular graph of girth $g$ that does not contain $(g+1)$-cycles} by $n(k,g,\underline{g+1})$, and talk about $(k,g,\underline{g+1})$-\textit{cages} (note that we allow for both even and odd $g$). 
This is analogous with the $n(k,g)$ notation used
in the context of the Cage Problem and denoting \textit{the order of a smallest $k$-regular graph of girth $g$}. 
Clearly, 
\begin{equation}\label{first}
n(k,g) \leq n(k,g,\underline{g+1}) ,
\end{equation}
for all $k \geq 3,g \geq 3$. Recalling again the observation stated at the end of the first paragraph, we note that 
at this point, we are not aware of any pair of parameters $(k,g)$ for 
which $g$ is odd and the above equality is not strict. However, we have not been able to prove that strict inequality is
universal. Furthermore, based on (a) the above observation that the canonical double cover of a $k$-regular graph of odd girth $g$ without $(g+1)$-cycles yields a $k$-regular graph of girth at least $g+3$ and (b) the monotonicity of the the function $ n(k,g) $ with respect to the parameter $g$~\cite{ExooR1}, we obtain another fundamental inequality tying together the 
`classical' cage orders and the orders of cages considered in our paper. More precisely,
\begin{equation}\label{second}
n(k,g+3) \leq 2n(k,g,\underline{g+1}) ,
\end{equation}
for all $k \geq 3 $ and odd $g \geq 3$. 

Despite the simplicity of the presented ideas, canonical double covers of $(k,g,\underline{g+1})$-cages yield graphs whose orders are surprisingly close to the 
orders of corresponding $(k,g+3)$-cages. For example, using the orders of the $(k,g,\underline{g+1})$-cages determined later in this paper (in Section~\ref{sec:computationalResults}) yields the following comparisons (where the fourth column labeled $2 n(k,g,\underline{g+1}) $ is the order of the canonical double cover of the $ (k,g,\underline{g+1}) $-cage):

\begin{table}[h!]
\centering
\begin{tabular}{||c|c|c|c||}
\hline \hline
degree $k$ & girth $g$ & $n(k,g+3)$ & $2 n(k,g,\underline{g+1}) $ \\
\hline
$3$ & $3$ & $14$ & $20$ \\
$3$ & $5$ & $ 30$ & $36$ \\
$3$ & $7$ & $70$  & $72$ \\
$3$ & $9$ & $126$ & $152$ \\
$4$ & $3$ & $26$ & $30$ \\
$4$ & $5$ & $80$ & $90 $ \\
$5$ & $3$ & $42$ & $52$ \\
$6$ & $3$ & $62$ & $68$ \\
\hline \hline 
\end{tabular}
\caption{A comparison between the orders of $(k,g+3)$-cages and the orders of the canonical double covers of $(k,g,\underline{g+1})$-cages.}
\end{table}
A lower bound derived later in our paper yields that $n(3,11,\underline{12}) \geq 144$ (by applying Propositions~\ref{prop:lowerBound} and~\ref{prop:strictlowerBound}). That means that
if a graph of this order existed, and we used the corresponding $(3,11,\underline{12})$-cage 
in the canonical double cover construction, the order of the canonical double cover would be of order
$284$ which is quite a bit smaller than the order $384$ of the smallest known $(3,14)$-graph~\cite{E02}. 
Even though we conjecture that no $(3,11,\underline{12})$-graph of order $144$ exists, there is still some
leeway. Specifically, any canonical double cover of a $(3,11,\underline{12})$-graph of order greater than
$144$ and smaller than $ 384 / 2 = 192$ would result in a $(3,14)$-graph of order smaller than the currently
best known $(3,14)$-graph of order $384$. 
Further inspection of the results listed in \cite{ExooR1} reveals that the orders of the smallest 
known cubic graphs of even girths greater than $14$ differ considerably from the corresponding Moore 
bounds. This might suggest that finding a canonical double cover of a
relatively small cubic graph of appropriate girth $g$ not containing $(g+1)$-cycles might repeatedly 
lead to smaller cubic graphs of girth $g+3$ than the so-called record graphs listed in \cite{ExooR1} 
(a \textit{record $(k,g)$-graph} is a smallest \textit{known} $(k,g)$-graph). This might
also be the case for larger degrees $k$ for which the factor of $2$ coming from the double cover
construction is less significant when compared to the increase between the orders of $k$-regular graphs 
of girths $g$ and $g+3$.

\section{Basic Concepts and Estimates}
Our focus is on $(k,g)$-graphs not containing $(g+1)$-cycles, the existence of which was established for example in~\cite{Pisanski}. A natural lower bound on the order 
 $n(k,g)$ of $(k,g)$-graphs is the \textit{Moore bound}, denoted by $M(k,g)$, and defined as
	
	\[
	M(k,g) =  \left\{
	\begin{array}{lr}
		\dfrac {k(k-1)^{(g-1)/2}-2}{k-2}, & g \mbox{ odd,}\\
		\dfrac {2(k-1)^{g/2}-2}{k-2}, & g \mbox{ even. }
	\end{array} 
	\right.
	\]

 \noindent
 Erd\H{o}s and Sachs \cite{Erds} proved using the canonical double cover construction for any given integer 
 pair $k\geq 3$ and \textit{odd} $g \geq 3$ that $$n(k,g+1) \leq 2n(k,g).$$
	
 \noindent 
 Their bound has been further improved in \cite{Babuena}:
	
 \[
		n(k,g+1) \leq  \left\{
		\begin{array}{lr}
			2n(k,g) - 2\dfrac {k(k-1)^{(g-3)/4}-2}{k-2}, & g\equiv 3\mod4 , \\
			2n(k,g) - 4\dfrac {(k-1)^{(g-1)/4}-1}{k-2}, &  g\equiv 1\mod4, 
		\end{array} 
		\right.
 \]
	by using a canonical double cover of excised graphs. 
	
	We begin by improving the lower bounds on the orders of the members of the more specialized class
 of graphs considered in our paper.
 Combining the Moore bound with \eqref{second} we obtain for odd $ g \geq 3 $
	\[ M(k,g+3)/2 \leq n(k,g+3)/2 \leq n(k,g,\underline{g+1}). \]

Sauer~\cite{S67} proved that for every $k \geq 2$ and $g \geq 3$, we have $n(k,g)<n(k,g+1)$. Later, a short proof of this fact was also given by Fu, Huang and Rodger~\cite{FHR97}. We modify their idea to prove:
 \begin{theorem}
 \label{th:monotonicityLike}
 For every $k \geq 2$ and $g \geq 3$, we have $n(k,g,\underline{g+1}) \leq n(k,g+2,\underline{g+3})-2$.
 \end{theorem}
 \begin{proof}
 The result is clear for $k=2$, so henceforth assume that $k\geq3$. Let $\Gamma$ be a $(k,g+2,\underline{g+3})$-cage (recall that $\Gamma$ is guaranteed to exist due to~\cite{Pisanski}). Let $\mathcal{C}$ be a cycle of length $g+2$ in $\Gamma$ containing the edges $ux_1$, $uv$ and $vy_1$. Let $\{v,x_1,x_2,\ldots,x_{k-1}\}$ be the set of neighbors of $u$ and let $\{u,y_1,y_2,\ldots,y_{k-1}\}$ be the set of neighbors of $v$. Let $E'=\{x_1y_1, x_2y_2, \ldots, x_{k-1}y_{k-1}\}$. Finally, let $\Sigma=\Gamma-u-v+E'$ (see Fig.~\ref{fig:GammaAndSigma}). We claim that $\Sigma$ is a $k$-regular graph on $|V(\Gamma)|-2$ vertices of girth $g$ without cycles of length $g+1$.

        \begin{figure}[h!]
\begin{center}
\begin{tikzpicture}[scale=0.55]

        
        \foreach \x in {0}{
        \draw[fill] (\x*360/8+22.5:4.5) circle (2.4pt);
        \draw[line width=1.5pt] ({360/8 * (\x+1)+22.5}:4.5) -- ({360/8 * (\x )+22.5}:4.5);
        }
        \draw[fill] (1*360/8+22.5:4.5) circle (2.4pt);
        \foreach \x in {2,...,7}{
        \draw[fill] (\x*360/8+22.5:4.5) circle (2.4pt);
        \draw[line width=1.5pt] ({360/8 * (\x+1)+22.5}:4.5) -- ({360/8 * (\x )+22.5}:4.5);
        }

        \node at (0.2,4.0) {$\cdots$};

        \draw({360/8 * (1)+22.5}:4.5) -- (3,4.5);
        \draw({360/8 * (1)+22.5}:4.5) -- (2.8,4.7);
        \node at (2.2,4.7) {$\cdots$};
        \draw({360/8 * (1)+22.5}:4.5) -- (1.6,4.7);

        \draw({360/8 * (2)+22.5}:4.5) -- (-1,4.5);
        \draw({360/8 * (2)+22.5}:4.5) -- (-1.2,4.7);
        \node at (-1.8,4.7) {$\cdots$};
        \draw({360/8 * (2)+22.5}:4.5) -- (-2.4,4.7);

        \draw({360/8 * (3)+22.5}:4.5) -- (-5,1.1);
        \draw({360/8 * (3)+22.5}:4.5) -- (-5.2,1.9);
        \node at (-5.1,1.4) {$\cdots$};
        \draw({360/8 * (3)+22.5}:4.5) -- (-5.4,1.7);

        \draw({360/8 * (0)+22.5}:4.5) -- (5,1.1);
        \draw({360/8 * (0)+22.5}:4.5) -- (5.2,1.9);
        \node at (5.1,1.4) {$\cdots$};
        \draw({360/8 * (0)+22.5}:4.5) -- (5.4,1.7);

        \draw({360/8 * (4)+22.5}:4.5) -- (-5,-1.0);
        \draw({360/8 * (4)+22.5}:4.5) -- (-5.2,-1.9);
        \node at (-5.1,-1.4) {$\cdots$};
        \draw({360/8 * (4)+22.5}:4.5) -- (-5.4,-1.7);

        \draw({360/8 * (7)+22.5}:4.5) -- (5,-1.0);
        \draw({360/8 * (7)+22.5}:4.5) -- (5.2,-1.9);
        \node at (5.1,-1.4) {$\cdots$};
        \draw({360/8 * (7)+22.5}:4.5) -- (5.4,-1.7);

        \draw({360/8 * (5)+22.5}:4.5) -- (-4.2,-2.4);
        \draw[fill] (-4.2,-2.4) circle (2.4pt);
        \node at (-4.2,-3.1) {$\vdots$};

        \draw({360/8 * (5)+22.5}:4.5) -- (-4.2,-4.0);
        \draw[fill] (-4.2,-4.0) circle (2.4pt);

        \draw({360/8 * (6)+22.5}:4.5) -- (4.2,-2.4);
        \draw[fill] (4.2,-2.4) circle (2.4pt);
        \node at (4.2,-3.1) {$\vdots$};

        \draw({360/8 * (6)+22.5}:4.5) -- (4.2,-4.0);
        \draw[fill] (4.2,-4.0) circle (2.4pt);

        \node at (5.1, -4.0) {$y_{k-1}$};
        \node at (4.7, -2.4) {$y_{2}$};
        \node at (3.7, -1.3) {$y_{1}$};

        \node at (-5.1, -4.0) {$x_{k-1}$};
        \node at (-4.7, -2.4) {$x_{2}$};
        \node at (-3.7, -1.3) {$x_{1}$};

        \node at (-1.7, -4.6) {$u$};
        \node at (1.7, -4.6) {$v$};
        
\end{tikzpicture} \quad
\begin{tikzpicture}[scale=0.55]

        
        \foreach \x in {0}{
        \draw[fill] (\x*360/8+22.5:4.5) circle (2.4pt);
        \draw[line width=1.5pt] ({360/8 * (\x+1)+22.5}:4.5) -- ({360/8 * (\x )+22.5}:4.5);
        }
        \draw[fill] (1*360/8+22.5:4.5) circle (2.4pt);
        
        \foreach \x in {2,...,3}{
        \draw[fill] (\x*360/8+22.5:4.5) circle (2.4pt);
        \draw[line width=1.5pt] ({360/8 * (\x+1)+22.5}:4.5) -- ({360/8 * (\x )+22.5}:4.5);
        }
        \draw[fill] (4*360/8+22.5:4.5) circle (2.4pt);
        
        \foreach \x in {7}{
        \draw[fill] (\x*360/8+22.5:4.5) circle (2.4pt);
        \draw[line width=1.5pt] ({360/8 * (\x+1)+22.5}:4.5) -- ({360/8 * (\x )+22.5}:4.5);
        }
        
        \node at (0.2,4.0) {$\cdots$};

        \draw({360/8 * (1)+22.5}:4.5) -- (3,4.5);
        \draw({360/8 * (1)+22.5}:4.5) -- (2.8,4.7);
        \node at (2.2,4.7) {$\cdots$};
        \draw({360/8 * (1)+22.5}:4.5) -- (1.6,4.7);

        \draw({360/8 * (2)+22.5}:4.5) -- (-1,4.5);
        \draw({360/8 * (2)+22.5}:4.5) -- (-1.2,4.7);
        \node at (-1.8,4.7) {$\cdots$};
        \draw({360/8 * (2)+22.5}:4.5) -- (-2.4,4.7);

        \draw({360/8 * (3)+22.5}:4.5) -- (-5,1.1);
        \draw({360/8 * (3)+22.5}:4.5) -- (-5.2,1.9);
        \node at (-5.1,1.4) {$\cdots$};
        \draw({360/8 * (3)+22.5}:4.5) -- (-5.4,1.7);

        \draw({360/8 * (0)+22.5}:4.5) -- (5,1.1);
        \draw({360/8 * (0)+22.5}:4.5) -- (5.2,1.9);
        \node at (5.1,1.4) {$\cdots$};
        \draw({360/8 * (0)+22.5}:4.5) -- (5.4,1.7);

        \draw({360/8 * (4)+22.5}:4.5) -- (-5,-1.0);
        \draw({360/8 * (4)+22.5}:4.5) -- (-5.2,-1.9);
        \node at (-5.1,-1.4) {$\cdots$};
        \draw({360/8 * (4)+22.5}:4.5) -- (-5.4,-1.7);

        \draw({360/8 * (7)+22.5}:4.5) -- (5,-1.0);
        \draw({360/8 * (7)+22.5}:4.5) -- (5.2,-1.9);
        \node at (5.1,-1.4) {$\cdots$};
        \draw({360/8 * (7)+22.5}:4.5) -- (5.4,-1.7);

        \draw(-4.2,-2.4) -- (4.2,-2.4);
        \draw[fill] (-4.2,-2.4) circle (2.4pt);
        \node at (-4.2,-3.1) {$\vdots$};

        \draw(-4.2,-4.0) -- (4.2,-4.0);
        \draw[fill] (-4.2,-4.0) circle (2.4pt);

        \draw(-4.2,-1.7) -- (4.2,-1.7);
        
        \draw[fill] (4.2,-2.4) circle (2.4pt);
        \node at (4.2,-3.1) {$\vdots$};

        \draw[fill] (4.2,-4.0) circle (2.4pt);

        \node at (5.1, -4.0) {$y_{k-1}$};
        \node at (4.7, -2.4) {$y_{2}$};
        \node at (3.7, -1.3) {$y_{1}$};

        \node at (-5.1, -4.0) {$x_{k-1}$};
        \node at (-4.7, -2.4) {$x_{2}$};
        \node at (-3.7, -1.3) {$x_{1}$};

        
\end{tikzpicture}
\end{center}
\caption{The graphs $\Gamma$ (left) and $\Sigma$ (right). Edges in $E(\mathcal{C})$ are marked in bold. }\label{fig:GammaAndSigma}
\end{figure}

 The order of $\Sigma$ is clearly $|V(\Gamma)|-2$. Since $g+2 \geq 5$, $u$ and $v$ do not have any common neighbors in $\Gamma$ and we also have $E(\Gamma) \cap E' = \emptyset$. Therefore, $\Sigma$ is a simple $k$-regular graph. Moreover, $\Sigma$ contains a cycle of length $g$, namely $\mathcal{C}-u-v+x_1y_1$. 
 We will now show that $\Sigma$ contains no cycle of length smaller than $g$ nor a cycle of length $g+1$. Let $\mathcal{C'}$ be any cycle in $\Sigma$. If $E(\mathcal{C'}) \cap E' = \emptyset$, then $\mathcal{C'}$ is also a cycle in $\Gamma$. This means that $|E(\mathcal{C'})| \geq g+2$, because $\Gamma$ has girth $g+2$. 
 If $|E(\mathcal{C'}) \cap E'| \geq 1$, then we distinguish between the following 2 cases:

 \begin{enumerate}
     \item \textbf{Case 1:} $\mathcal{C'}$ contains a subpath $P$ between $x_i$ and $x_j$, or symmetrically between $y_i$ and $y_j$, such that $E(P) \cap E' = \emptyset$ for some $i,j \in \{1,2,\ldots,k-1\}, i \neq j$, and there is an edge $e_i \in E' \cap E(\mathcal{C}')$ incident with $x_i$ (or $y_i$, respectively) and $e_j \in E' \cap E(\mathcal{C}')$ incident with $x_j$ (or $y_j$, respectively). If $|E(P)|<g$, we obtain a contradiction because then $P+x_iu+x_ju$ (or $P+y_iu+y_ju$, respectively) would be a cycle of length strictly smaller than $g+2$ in $\Gamma$. So $|E(P)| \geq g$, and therefore $|E(\mathcal{C}')| \geq g+2$. 
     \item \textbf{Case 2:} $\mathcal{C'}$ contains a subpath $P$ between $x_i$ and $y_j$ such that $E(P) \cap E' = \emptyset$ for some $i,j \in \{1,2,\ldots,k-1\}$ (note that $i=j$ is allowed) and there is an edge $e_i \in E' \cap E(\mathcal{C}')$ incident with $x_i$ and $e_j \in E' \cap E(\mathcal{C}')$ incident with $y_j$. If $|E(P)|<g-1$, we obtain a contraction because then $P+x_iu+uv+vy_j$ would be a cycle of length strictly smaller than $g+2$ in $\Gamma$. If $|E(P)|=g-1$ and $|E(\mathcal{C'}) \cap E'|=1$, then $|E(\mathcal{C'})|=g$ and such cycles are allowed in $\Sigma$. If $|E(P)|=g-1$ and $|E(\mathcal{C'}) \cap E'|>1$, then $i \neq j$ and $\mathcal{C'}$ consists of $P$, the edges $e_i$ and $e_j$ and a path between $y_i$ and $x_j$ containing at least one edge. Therefore $|E(\mathcal{C}')| \geq g-1+1+1+1=g+2$. 
     If $|E(P)|=g$ we obtain a contradiction, because $P+x_iu+uv+vy_j$ would be a cycle of length $g+3$ in $\Gamma$. Finally, if $|E(P)| \geq g+1$, then $|E(\mathcal{C}')| \geq g+2$ and such cycles are allowed in $\Sigma$.
 \end{enumerate}
 \end{proof}
 
To conclude the section, note that the above proof took advantage of the existence of a single cycle $\mathcal{C}$ of length $g+2$. Should $\Gamma$ contain several $(g+2)$-cycles which are `sufficiently far' one from another, one could
apply the above vertex pair removal repeatedly and obtain even smaller $(k,g,\underline{g+1})$-graphs. 
Even though it appears likely that $(k,g+2,\underline{g+3})$-cages might contain several such cycles, we do not 
explore this possibility beyond the result stated in Theorem~\ref{th:monotonicityLike}.

		\section{Lower bound on $n(k,2t+1,\underline{2t+2})$}

  In this section, we aim to establish a lower bound on the order of $(k,2t+1,\underline{2t+2})$-graphs. Our main idea stems from the folklore knowledge that every vertex $v$ of a \((k, 2t+1)\)-graph $\Gamma$ is the root of a subgraph of $\Gamma$
  we shall call the \textit{Moore tree rooted at $v$} and denote by $\mathcal{T}_v(k,2t+1)$. This is a $k$-ary tree in which the distance between the root $v$ and any leaf is equal to $t$~\cite{ExooR1}. This can be seen by combining the fact that there cannot be an edge between two vertices for which the sum of the distances to $v$ is smaller than $2t$ and the fact that $\Gamma$ is $k$-regular. Generally, the subgraph \textit{induced} in $\Gamma$ by the vertices of this tree does not 
  necessarily have to be a tree. 
  Nevertheless, since the girth of $\Gamma$ is assumed to be equal to $2t+1$, $\Gamma$ must 
  contain at least one vertex $v$ such that the subgraph induced by the vertices of $\mathcal{T}_v(k,2t+1)$ is not a 
  tree, but contains edges incident to a pair of leaves of $\mathcal{T}_v(k,2t+1)$ (since this is the only way in which $v$ can be included in a cycle of length $2t+1$). We refer to the edges connecting 
  two leaves of such a Moore tree as \textit{horizontal edges}. Let \(S\) denote the set of leaves of such a Moore tree $\mathcal{T}_v(k,2t+1)$, and let \(E(S)\) denote the set of horizontal edges connecting the leaves of $\mathcal{T}_v(k,2t+1)$. In what follows, we will suppress the subscript $v$ as the actual root is not important
  (beyond the fact that we assume that \(E(S)\) is not empty). 
	
	\begin{lemma}\label{lem2}
 Let \(t \geq 1\) and let \(\Gamma\) be a $(k,2t+1,\underline{2t+2})$-graph. Then the number of horizontal edges of any Moore tree $\mathcal{T}(k,2t+1)$ contained in $\Gamma$ is bounded from above by the following inequality:
\[
|E(S)| \leq \frac{k(k - 1)^{t-1}}{2}.
\]
	\end{lemma}
	
	\begin{proof}
 Assume that $\mathcal{T}(k,2t+1)$ is a Moore tree occurring as a subgraph of $\Gamma$ rooted at a vertex \(v\), and \(S\) denotes the leaves of $\mathcal{T}(k,2t+1)$. The number of vertices in \(S\) is given by:
    \[
    |S| = k(k - 1)^{t-1}.
    \]
Each vertex in \(S\) can be incident to at most one horizontal edge; otherwise, the two horizontal edges sharing a
vertex in \(S\) would form a cycle of length \(2t + 2\). To see this, suppose a vertex \(x \in S\) is connected by horizontal edges $xu_1$ and $xu_2$ to two different vertices \(u_1\) and \(u_2\) in \(S\). The cycle \(v \ldots u_1 x u_2 \ldots v\) 
formed by the shortest path from $v$ to $u_1$ followed by $u_1 x$, $x u_2 $, and the shortest path between $u_2$ and 
$v$ would have length \(2t + 2\), contradicting the requirement that $\Gamma$ has no  \((2t+2)\)-cycle (see Fig.~\ref{fig:exampleK3T2} for an example of this forbidden situation where $k=3$ and $t=2$). Since each vertex in \(S\) can be incident to at most one horizontal edge, the maximum number of horizontal edges is obtained by pairing vertices, which yields the desired bound:
    \[
    |E(S)| \leq \frac{|S|}{2} = \frac{k(k - 1)^{t-1}}{2}.
    \]

    \begin{figure}[ht]
        \centering
        \begin{tikzpicture}[scale=1.00]
        
        \draw (0,0)--(-2,-1);
        \draw[line width=1.5pt] (0,0)--(0,-1);
        \draw[line width=1.5pt] (0,0)--(2,-1);
        \draw (-2,-1)--(-2.6,-2);
        \draw (-2,-1)--(-1.4,-2);
        \draw[line width=1.5pt] (0,-1)--(-0.6,-2);
        \draw (0,-1)--(0.6,-2);
        \draw[line width=1.5pt] (2,-1)--(1.4,-2);
        \draw (2,-1)--(2.6,-2);
        
        \draw[fill] (0,0) circle (1.8pt);
        \draw[fill] (-2,-1) circle (1.8pt);
        \draw[fill] (0,-1) circle (1.8pt);
        \draw[fill] (2,-1) circle (1.8pt);
        \draw[fill] (-2.6,-2) circle (1.8pt);
        \draw[fill] (-1.4,-2) circle (1.8pt);
        \draw[fill] (-0.6,-2) circle (1.8pt);
        \draw[fill] (0.6,-2) circle (1.8pt);
        \draw[fill] (1.4,-2) circle (1.8pt);
        \draw[fill] (2.6,-2) circle (1.8pt);

        \draw[line width=1.5pt] (-2.6,-2) .. controls (-1.6,-2.3) and (-1.6,-2.3) .. (-0.6, -2);
        \draw[line width=1.5pt] (-2.6,-2) .. controls (-0.6,-2.7) and (-0.6,-2.7) .. (1.4, -2);
        
        \node at (-2.8, -1.7) {$x$};
        \node at (-0.8, -1.7) {$u_1$};
        \node at (1.2, -1.7) {$u_2$};
        \node at (-0.3, 0.1) {$v$};
        \end{tikzpicture}
        \caption{An example showing that a $(3,5,\underline{6})$ graph cannot have a vertex $x \in S$ connected by horizontal edges $xu_1$ and $xu_2$ to two different vertices $u_1$ and $u_2$ in $S$. The bold edges indicate a $6$-cycle.}\label{fig:exampleK3T2}
        \end{figure}
\end{proof}

Let $E'(S)$ be the set of edges incident to the leaves of $\mathcal{T}(k,2t+1)$, which are different from horizontal edges and different from the edges in $\mathcal{T}(k,2t+1)$; these are the edges connecting the leaves of $\mathcal{T}(k,2t+1)$ to 
vertices outside $\mathcal{T}(k,2t+1)$. We now prove the following proposition which can be viewed as an analogue of 
the Moore bound for classical cages.

	\begin{prop}\label{prop:lowerBound}
Let $ k\geq 3 $ and $ t \geq 1 $. Then $n(k,2t+1,\underline{2t+2})$ is bounded from below as follows:
\[
n(k,2t+1,\underline{2t+2}) \geq M(k, 2t+1) + |E'(S)| \geq M(k, 2t+1) + (k-2)k(k-1)^{t-1}.
\]
 
	\end{prop}
	
	\begin{proof}

Let \(\Gamma\) be a $(k,2t+1,\underline{2t+2})$-graph containing the Moore tree $\mathcal{T}(k,2t+1)$ rooted at vertex $v$ as a subgraph. Note that two leaves $u_1$ and $u_2$ of $\mathcal{T}(k,2t+1)$ cannot have a common neighbor $x \in V(\Gamma) \setminus V(\mathcal{T}(k,2t+1))$, as that would give rise to a $ (g +1) $-cycle \(v \ldots u_1 x u_2 \ldots v\); where again $ v \ldots u_1 $
denotes the shortest path from $v$ to $u_1$ and $ u_2 \ldots v $ denotes the shortest path between $u_2$ and $v$. Since $|V(\mathcal{T}(k,2t+1))|=M(k,2t+1)$ and every edge in $E'(S)$ is incident
with a different vertex not contained in $ \mathcal{T}(k,2t+1)$, we obtain:
\[
n(k,2t+1,\underline{2t+2}) \geq M(k, 2t+1) + |E'(S)|.
\]
Since every vertex in $\Gamma$ has degree $k$, $|S|=k(k-1)^{t-1}$ and $|E'(S)|=|S|(k-1)-2|E(S)|$, applying Lemma~\ref{lem2} yields
\[
|E'(S)| \geq k(k-1)^{t} -2 \frac{k(k - 1)^{t-1}}{2},
\]
and therefore
\[
n(k,2t+1,\underline{2t+2}) \geq M(k,2t+1)+(k-2)k(k-1)^{t-1}.
\]
\end{proof}	

Note that a $(k,2t+1,\underline{2t+2})$-graph $\Gamma$ of order $M(k, 2t+1) + (k-2)k(k-1)^{t-1}$
contains no other vertices but the vertices of $\mathcal{T}(k,2t+1)$ and the vertices not contained in $\mathcal{T}(k,2t+1)$ which
are incident with the leaves of $\mathcal{T}(k,2t+1)$ (each through a unique edge in $E'(S)$).
This means that the subgraph of $\Gamma$ induced by the $(k-2)k(k-1)^{t-1}$ vertices not contained 
in $\mathcal{T}(k,2t+1)$ is a $(k-1)$-regular graph of girth at least $2t+1$ not containing a
$(2t+2)$-cycle. Hence, equality in the above bound yields the existence of a $(k-1)$-regular
graph of order$(k-2)k(k-1)^{t-1}$, girth at least $2t+1$, and not containing a $(2t+2)$-cycle.
Consulting Table~\ref{tab:overviewExactComputations}, we observe that in all the resolved cases
listed in there, the order of the cage exceeds the lower bound of Proposition~\ref{prop:lowerBound}.
In particular, $n(3,2t+1,\underline{2t+2}) > M(3, 2t+1) + 3\cdot 2^{t-1}$, for $ t=1,2,3,$ and $4$,
and $n(4,2t+1,\underline{2t+2}) > M(4, 2t+1) + 2 \cdot 4 \cdot 3^{t-1} $ for $ t=1 $ and $2$.
As argued above, the equality $n(4,2t+1,\underline{2t+2}) = M(4, 2t+1) + 2 \cdot 4 \cdot 3^{t-1} $
in case of $t=3$ would yield the existence of a $(3,7,\underline{8})$-graph or a $(3,9)$-graph of order $ 2 \cdot 4 \cdot 3^2 = 72 $. Since $n(3,7,\underline{8})= 36$ by Table~\ref{tab:overviewExactComputations} and $n(3,9)= 58$ \cite{ExooR1}, this 
equality may still be satisfied for $k=4$ and $t=3$.
In the last result of this section, we present a necessary condition for the above inequality to be strict (from which one can derive that equality for $k=4$ and $t=3$ is in fact not attainable).

	\begin{prop}\label{prop:strictlowerBound}
    Let $ k\geq 3 $ and $ t \geq 1 $. If $\Gamma$ is a $(k,2t+1,\underline{2t+2})$-graph of order 
    $M(k, 2t+1) + (k-2)k(k-1)^{t-1}$, then $4t+2$ must divide $(M(k, 2t+1) + (k-2)k(k-1)^{t-1})k(k-1)^{t-1}$.

	\end{prop}
	
	\begin{proof}
  Let $\Gamma$ be a a $(k,2t+1,\underline{2t+2})$-graph of order $M(k, 2t+1) + (k-2)k(k-1)^{t-1}$. 
  A close inspection of the proof of Proposition~\ref{prop:lowerBound} yields that $|E(S)| = \frac{k(k - 1)^{t-1}}{2} $
  for all $ v \in V(\Gamma) $ and their corresponding Moore trees $\mathcal{T}_v(k,2t+1)$. Note that each 
  horizontal edge in $E(S)$ yields a girth cycle of length $2t+1$ through $v$; with all such cycles different
  and all the girth cycles through $v$ of this form. Hence, each $ v \in V(\Gamma)$ is contained in exactly
  $ \frac{k(k - 1)^{t-1}}{2} $ cycles of length $2t+1$. An easy double-counting of pairs formed by a vertex and
  a $(2t+1)$-cycle containing it gives the following identity 
  \[ |V(\Gamma) | \cdot \frac{k(k - 1)^{t-1}}{2} = r \cdot (2t+1) ,\]
  where $r$ is the (integral) total count of $(2t+1)$-cycles in $ \Gamma $. By our assumption, 
  $ |V(\Gamma) | = M(k, 2t+1) + (k-2)k(k-1)^{t-1} $, and therefore 
   \[ (M(k, 2t+1) + (k-2)k(k-1)^{t-1}) \cdot \frac{k(k - 1)^{t-1}}{2} = r \cdot (2t+1) . \]
   The divisibility condition now immediately follows.
   \end{proof}
   Applying Proposition~\ref{prop:strictlowerBound} to the pair $k=4$ and $t=3$ discussed just above its statement 
   yields that the order of a $(4,7,\underline{8})$-cage must be strictly larger
   than the bound given by Proposition~\ref{prop:lowerBound}.

\section{Exhaustive Generation of $(k,g,\underline{g+1})$-Graphs}
\label{sec:algo}
We propose an algorithm that can exhaustively enumerate all (pairwise non-isomorphic) $(k,g,\underline{g+1})$-graphs on $n$ vertices for given integers $k, g$ and $n$. Such algorithms have been successfully used for finding extremal graphs for problems related to the Cage Problem~\cite{GJ24,GJV24,JJP24,MMN98}. Our algorithm is based on similar ideas as these algorithms.

One observation used by our algorithm is that every $(k,g,\underline{g+1})$-graph $\Gamma$ (in fact every $(k,g)$-graph) contains a $(k,g)$ Moore tree $\mathcal{T}(k,g)$ as a subgraph. Hence, the graph obtained by adding $n-|\mathcal{T}(k,g)|=n-M(k,g)$ isolated vertices to the $(k,g)$ Moore tree occurs as a subgraph of any $(k,g,\underline{g+1})$-graph on $n$ vertices. The algorithm recursively adds edges to this graph (one edge per each step) in all possible ways to generate all $(k,g,\underline{g+1})$-graphs. In order to speed up the algorithm, several additional optimizations are used:
\begin{enumerate}
    \item Often, adding one set of edges results in a graph isomorphic with the graph obtained by adding a different set of edges. In such cases, the algorithm will only continue adding edges to the first graph and prunes the second isomorphic graph from the recursion tree. In order to verify whether the recursive function was already called before with as argument a graph which is isomorphic with the current graph, the algorithm computes a canonical form of the graph using the \textit{nauty} package~\cite{MP14} (two graphs are isomorphic if they have the same canonical form).
    \item Although computing the previously mentioned canonical form is computationally expensive, it is typically worthwhile as it prevents a lot of duplicate computations. Nevertheless, we aim to avoid computing the canonical form when it is clear that isomorphic copies will arise. To achieve this, the algorithm keeps track of the set of all isolated vertices and leverages the fact that adding an edge $uv$ or $uw$ (for any vertex $u$) results in isomorphic graphs when $v$ and $w$ are isolated. Therefore, the algorithm only explores the first option and does not explicitly calculate the canonical form.
    \item Every graph containing a cycle of length smaller than $g$ can be pruned.
    \item Every graph containing a cycle of length $g+1$ can be pruned.
    \item For each vertex, the algorithm keeps track of which edges are eligible to be added incident with this vertex (without resulting in a graph that will be pruned). If there is a vertex which has too few eligible edges (i.e., it cannot obtain degree $k$) the current graph is pruned as it will never result in a $(k,g,\underline{g+1})$-graph.
    \item In each recursion step the algorithm considers precisely one edge of the graph and creates two branches corresponding to whether or not the edge was added. In order to restrict the search space as soon as possible, the algorithm considers the edge incident with the vertex which has the fewest eligible incident edges available.
    \item The algorithm uses appropriate data structures to efficiently support the above operations. For example, the canonical forms are stored in a data structure that allows for fast insertion and lookup and all small sets are stored as bitsets, which support all standard set operations using hardware-friendly operations. 
\end{enumerate}

The pseudocode of this algorithm is shown in Algorithm~\ref{algo:genAlgo} (the main algorithm) and Algorithm~\ref{algo:recAddEdges} (the function that recursively adds the edges).

\begin{algorithm}[ht!]
\caption{generateAllGraphs(Integer $n$, Integer $k$, Integer $g$)}
\label{algo:genAlgo}
  \begin{algorithmic}[1]
		\STATE // All $(k,g,\underline{g+1})$-graphs will be generated by this function
		\IF{$n$ is smaller than any known lower bound}
			\RETURN
            \ENDIF
		\STATE $\Gamma \gets \text{disjoint union of the Moore tree }\mathcal{T}(k,g)\text{ and }n-M(k,g)\text{ isolated vertices}$
		\STATE $\text{eligibleEdges} \gets \text{calculateEligibleEdges}(\Gamma)$
		\STATE $\text{addEachEdgeRecursively}(n,k,g,\Gamma,\text{eligibleEdges})$
  \end{algorithmic}
\end{algorithm}

\begin{algorithm}[ht!]
\caption{addEachEdgeRecursively(Integer $n$, Integer $k$, Integer $g$, Graph $\Gamma=(V,E)$, Set eligibleEdges)}
\label{algo:recAddEdges}
  \begin{algorithmic}[1]
		\STATE // Consider one edge of $\Gamma$ in every recursive step and branch
        \IF{A pruning rule from Section~\ref{sec:algo} is applicable}
		  \RETURN
        \ENDIF
        \STATE
		\STATE // Every vertex has degree $k$, no more edges should be added
		\IF{$|E|=\frac{nk}{2}$}
			\IF{$\Gamma$ is a $(k,g,\underline{g+1})$-graph}
				\STATE Output $\Gamma$
			\ENDIF
			\RETURN
        \ENDIF
        \STATE
        \STATE // Add next edge incident with the vertex that has the fewest eligible edges
        \STATE $u \gets \arg\min_{v \in V(\Gamma), deg_{\Gamma}(v)<k}(|\{e\text{ is incident with }v\text{ and }e \in \text{eligibleEdges}\}|)$
        \STATE $e \gets \text{arbitrary edge from eligibleEdges incident with }u$
        \STATE 
        \STATE // Recursion branch 1: add this edge to $\Gamma$
        \STATE $\Gamma' \gets (V,E\cup\{e\})$
        \STATE $\text{newEligibleEdges} \gets \text{update}(\text{eligibleEdges},\Gamma')$
        \STATE $\text{addEachEdgeRecursively}(n,k,g,\Gamma',\text{newEligibleEdges})$
        \STATE 
        \STATE // Recursion branch 2: do not add this edge to $\Gamma$
        \STATE $\text{newEligibleEdges} \gets \text{eligibleEdges}\setminus\{e\}$
        \STATE $\text{addEachEdgeRecursively}(n,k,g,\Gamma,\text{newEligibleEdges})$
  \end{algorithmic}
\end{algorithm}

\section{Computational Results}
\label{sec:computationalResults}
\subsection{Determining $n(k,g,\underline{g+1})$ Using Algorithm~\ref{algo:genAlgo}}

We implemented\footnote{We make the source code and data related to this paper publicly available at:\\\url{https://github.com/JorikJooken/KGNoGPlus1Graphs}.} Algorithm~\ref{algo:genAlgo} and executed it on a supercomputer for various pairs $(k,g)$ to determine $n(k,g,\underline{g+1})$. The total CPU-time of all computations in this paper was around 9 CPU-weeks. More specifically, if the algorithm terminates with parameters $k, g, n'$ for all $n'<n$ without finding any $(k,g,\underline{g+1})$-graphs, then $n(k,g,\underline{g+1}) \geq n$, because the algorithm is exhaustive. On the other hand, if the algorithm finds any $(k,g,\underline{g+1})$-graph on $n$ vertices then clearly $n(k,g,\underline{g+1}) \leq n$ (regardless of whether the algorithm terminates within the given time). This results in the values of $n(k,g,\underline{g+1})$ as shown in Table~\ref{tab:overviewExactComputations}.

\begin{table}[h!]
\centering
\begin{tabular}{||c|c|c|c|c|c||}
\hline \hline
degree $k$ & girth $g$ & Proposition~\ref{prop:lowerBound} & $n(k,g,\underline{g+1}) $ & \# $(k,g,\underline{g+1})$-cages & \# vertex orbits\\
\hline
$3$ & $3$ & $7$ & $10$ & 2 & 3 and 3\\
$3$ & $5$ & $16$ & $18$ & 1 & 2\\
$3$ & $7$ & $34$ & $36$ & 1 & 2\\
$3$ & $9$ & $70$ & $76$ & 1 & 2\\
$4$ & $3$ & $13$ & $15$ & 2 & 1 and 3\\
$4$ & $5$ & $41$ & $45$ & 1 & 2\\
$5$ & $3$ & $21$ & $26$ & 2 & 4 and 15\\
$6$ & $3$ & $31$ & $34$ & 1 & 3\\
\hline \hline 
\end{tabular}
\caption{A summary of $(k,g,\underline{g+1})$-cages}
\label{tab:overviewExactComputations}
\end{table}

For all $(k,g)$-pairs shown in this table, we were able to determine $n(k,g,\underline{g+1})$ exactly as well as determine all $(k,g,\underline{g+1})$-graphs achieving this order. We make all these graphs publicly available at the House of Graphs~\cite{CDG23}; they can be found by searching for the term “without $(g+1)$-cycles”. We also searched for $(3,11,\underline{12})$-graphs, $(4,7,\underline{8})$-graphs, $(5,5,\underline{6})$-graphs and $(6,5,\underline{6})$-graphs (without waiting for the algorithm to terminate), but were unable to find any such graphs using Algorithm~\ref{algo:genAlgo}. As is clear from this table, the lower bound $M(k, 2t+1) + (k-2)k(k-1)^{t-1}$ that we determined in the current paper is not sharp in any of the cases but not very far off the exact value $n(k,g,\underline{g+1})$ either. In all of these cases, there are at most two non-isomorphic $(k,g,\underline{g+1})$-cages. The $(3,g,\underline{g+1})$-cages have between 2 and 3 vertex orbits; the $(3,3,\underline{4})$-cages, $(3,5,\underline{6})$-cages and $(3,7,\underline{8})$-cages are shown in Fig.~\ref{fig:33No4},~\ref{fig:35No6} and~\ref{fig:37No8}, respectively. Some of these graphs are well-known, for example the tricorn is a $(3,3,\underline{4})$-cage and the generalized Petersen graph $GP(9,2)$ is the unique $(3,5,\underline{6})$-cage.

        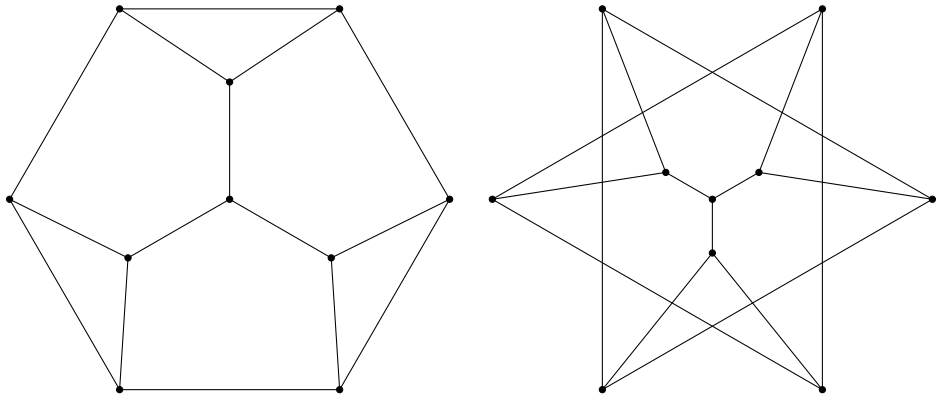
\begin{figure}[h!]
\begin{center}
\begin{tikzpicture}[scale=0.65]

         \foreach \x in {0,1,2}{
        \draw[fill] (\x*360/3-30:2.4) circle (1.8pt);
        \draw (\x*360/3-30:2.4) -- ({360/6 * (2*\x )}:4.5);
        \draw (\x*360/3-30:2.4) -- ({360/6 * (2*\x-1)}:4.5);
        \draw (\x*360/3-30:2.4) -- (0,0);
        }
        \draw[fill] (0,0) circle (1.8pt);
        \foreach \x in {0,1,...,5}{
        \draw[fill] (\x*360/6:4.5) circle (1.8pt);
        \draw ({360/6 * (\x+1)}:4.5) -- ({360/6 * (\x )}:4.5);
        }
        
\end{tikzpicture} \quad
\begin{tikzpicture}[scale=0.65]

         \foreach \x in {0,1,2}{
        \draw[fill] (\x*360/3-30+180:1.1) circle (1.8pt);
        \draw (\x*360/3-30+180:1.1) -- ({360/6 * (2*\x+2)}:4.5);
        \draw (\x*360/3-30+180:1.1) -- ({360/6 * (2*\x+3)}:4.5);
        \draw (\x*360/3-30+180:1.1) -- (0,0);
        }
        \draw[fill] (0,0) circle (1.8pt);
        \foreach \x in {0,1,...,5}{
        \draw[fill] (\x*360/6:4.5) circle (1.8pt);
        \draw ({360/6 * (\x)}:4.5) -- ({360/6 * (\x+2)}:4.5);
        }
        
\end{tikzpicture}
\end{center}
\caption{The two $(3,3,\underline{4})$-cages on 10 vertices. The leftmost graph is also known as the tricorn.}\label{fig:33No4}
\end{figure}

        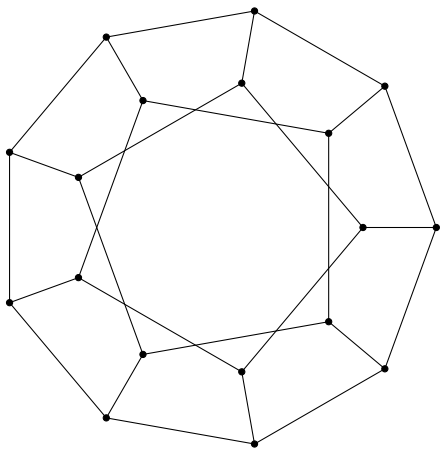
\begin{figure}[ht]
        \centering
        \begin{tikzpicture}[scale=0.65]
        \foreach \x in {0,1,...,8}{
        \draw[fill] (\x*360/9:4.5) circle (1.8pt);
        \draw[fill] (\x*360/9:3) circle (1.8pt);
        \draw (\x*360/9:3)--(\x*360/9:4.5);
        \draw (\x*360/9+360/9:4.5)--(\x*360/9:4.5);
        \draw (\x*360/9+720/9:3)--(\x*360/9:3);
        }
        \end{tikzpicture}
        \caption{The unique (3,5,\underline{6})-cage (the generalised Petersen graph GP(9,2)).}\label{fig:35No6}
        \end{figure}

\begin{figure}[h!]
\begin{center}
\begin{tikzpicture}[scale=0.8]
  \def\sides{9}
  \def\radius{3}

\begin{scope}[shift={(-4.0,-4.0)}]
  \foreach \i in {1,...,\sides} {
    \draw ({360/\sides * (\i + 1)}:\radius) --
    ({360/\sides * (\i)}:\radius);
  }
  \fill[red] ({360/\sides * 1}:\radius) circle (3pt);
  \fill[blue] ({360/\sides * 2}:\radius) circle (3pt);
  \fill[brown] ({360/\sides * 3}:\radius) circle (3pt);
  \fill[yellow] ({360/\sides * 4}:\radius) circle (3pt);
  \fill[black] ({360/\sides * 5}:\radius) circle (3pt);
  \fill[green] ({360/\sides * 6}:\radius) circle (3pt);
  \fill[cyan] ({360/\sides * 7}:\radius) circle (3pt);
  \fill[orange] ({360/\sides * 8}:\radius) circle (3pt);
  \fill[gray] ({360/\sides * 9}:\radius) circle (3pt);
\end{scope}

\begin{scope}[shift={(4.0,-4.0)}]
  \foreach \i in {1,...,\sides} {
    \draw ({360/\sides * (\i + 1)}:\radius) --
    ({360/\sides * (\i)}:\radius);
  }
  \fill[cyan] ({360/\sides * 1}:\radius) circle (3pt);
  \fill[black] ({360/\sides * 2}:\radius) circle (3pt);
  \fill[brown] ({360/\sides * 3}:\radius) circle (3pt);
  \fill[red] ({360/\sides * 4}:\radius) circle (3pt);
  \fill[orange] ({360/\sides * 5}:\radius) circle (3pt);
  \fill[green] ({360/\sides * 6}:\radius) circle (3pt);
  \fill[yellow] ({360/\sides * 7}:\radius) circle (3pt);
  \fill[blue] ({360/\sides * 8}:\radius) circle (3pt);
  \fill[gray] ({360/\sides * 9}:\radius) circle (3pt);
\end{scope}

\begin{scope}[shift={(0,5)}]
  \foreach \i in {1,...,\sides} {
    \draw ({360/\sides * (\i + 1)}:\radius) --
    ({360/\sides * (\i)}:\radius);
  }
  \fill[red] ({360/\sides * 1}:\radius) circle (3pt);
  \fill[green] ({360/\sides * 2}:\radius) circle (3pt);
  \fill[blue] ({360/\sides * 3}:\radius) circle (3pt);
  \fill[cyan] ({360/\sides * 4}:\radius) circle (3pt);
  \fill[brown] ({360/\sides * 5}:\radius) circle (3pt);
  \fill[orange] ({360/\sides * 6}:\radius) circle (3pt);
  \fill[yellow] ({360/\sides * 7}:\radius) circle (3pt);
  \fill[gray] ({360/\sides * 8}:\radius) circle (3pt);
  \fill[black] ({360/\sides * 9}:\radius) circle (3pt);
\end{scope}

\fill[red] (-1,-1) circle (3pt);
\fill[green] (-1,0) circle (3pt);
\fill[blue] (-1,1) circle (3pt);

\fill[cyan] (0,-1) circle (3pt);
\fill[brown] (0,0) circle (3pt);
\fill[orange] (0,1) circle (3pt);

\fill[yellow] (1,-1) circle (3pt);
\fill[gray] (1,0) circle (3pt);
\fill[black] (1,1) circle (3pt);

\node at (-1.3, 0.8) {$a_0$};
\node at (-0.3, 0.8) {$a_1$};
\node at (0.7, 0.8) {$a_2$};

\node at (-1.3, -0.2) {$a_3$};
\node at (-0.3, -0.2) {$a_4$};
\node at (0.7, -0.2) {$a_5$};

\node at (-1.3, -1.2) {$a_6$};
\node at (-0.3, -1.2) {$a_7$};
\node at (0.7, -1.2) {$a_8$};

\node at (-1.3, 7.0) {$b_0$};%
\node at (0.4, 7.4) {$b_3$};%
\node at (1.7, 6.9) {$b_6$};%
\node at (2.6, 5.0) {$b_2$};%
\node at (1.8, 3.3) {$b_5$};%
\node at (0.4, 2.5) {$b_8$};
\node at (-1.3, 2.7) {$b_1$};%
\node at (-2.2, 4.1) {$b_4$};%
\node at (-2.2, 5.8) {$b_7$};

\node at (2.7, -2.0) {$c_4$};
\node at (4.4, -1.6) {$c_2$};
\node at (5.7, -2.1) {$c_7$};
\node at (6.6, -4.0) {$c_5$};
\node at (5.8, -5.7) {$c_0$};
\node at (4.4, -6.5) {$c_8$};
\node at (2.7, -6.3) {$c_3$};
\node at (1.8, -4.9) {$c_1$};
\node at (1.8, -3.2) {$c_6$};

\node at (-5.3, -2.0) {$d_4$};
\node at (-3.6, -1.6) {$d_0$};
\node at (-2.3, -2.1) {$d_6$};
\node at (-1.4, -4.0) {$d_5$};
\node at (-2.2, -5.7) {$d_1$};
\node at (-3.6, -6.5) {$d_7$};
\node at (-5.3, -6.3) {$d_3$};
\node at (-6.2, -4.9) {$d_2$};
\node at (-6.2, -3.2) {$d_8$};
\end{tikzpicture}
\end{center}
\caption{The unique (3,7,\underline{8})-cage, wherein the 9 coloured vertices in the middle are adjacent to the vertices on the three different cycles with the same colour (these edges are not drawn). In other words, the graph has edge set $\{a_iz_i~|~z \in \{b,c,d\},i \in \{0,1,\ldots,8\}\} \cup \{z_iz_{i+1}~|~z \in \{b,c,d\},i \in \{0,1,\ldots,8\}\}$ (indices are taken modulo 9).}\label{fig:37No8} 
\end{figure}

One of the two $(4,3,\underline{4})$-cages is vertex-transitive (it is again well-known, namely it is the line graph of the Petersen graph), whereas the other one has 3 vertex orbits (see Fig.~\ref{fig:43No4}). The $(5,3,\underline{4})$-cages illustrate that $(k,g,\underline{g+1})$-cages are not necessarily very symmetrical: one of the graphs has 15 vertex orbits and its automorphism group only contains two automorphisms. 

We also point out that for the $(k,g)$-cage problem, it is known that $n(k,g)<n(k,g+1)$~\cite{S67}, whereas for the $(k,g,\underline{g+1})$-cage problem the values in the table indicate that $n(k,g,\underline{g+1})>n(k,g+1,\underline{g+2})$ is possible when $g$ is odd. In fact, for the cases in Table~\ref{tab:overviewExactComputations} this inequality holds for \textit{all} odd $g$; which can be deduced from the fact that all corresponding $(k,g+1)$-cages are
known to be bipartite \cite{ExooR1}, and therefore $n(k,g+1,\underline{g+2})=n(k,g+1)$ for these. However, as we demonstrated in Theorem~\ref{th:monotonicityLike}, $n(k,g,\underline{g+1}) \leq n(k,g+2,\underline{g+3})-2$. Finally, we stress that these orders also illustrate the need for the specialized algorithm discussed in this paper. For example, there are around $7.59 \cdot 10^{25}$ pairwise non-isomorphic connected 5-regular graphs on 26 vertices~\cite{OEIS} and we show that only two of these graphs are $(5,3,\underline{4})$-graphs.


        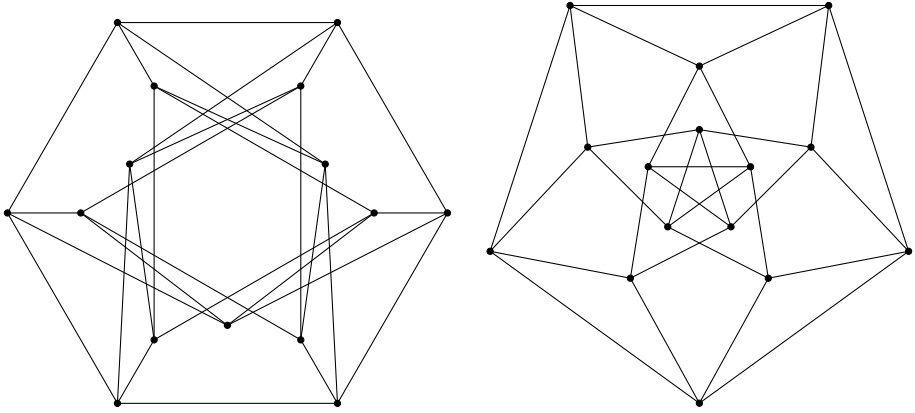
\begin{figure}[h!]
\begin{center}
\begin{tikzpicture}[scale=0.65]
        \foreach \x in {0,1,...,5}{
        \draw[fill] (\x*360/6:4.5) circle (1.8pt);
        \draw[fill] (\x*360/6:3) circle (1.8pt);

        \draw ({360/6 * (\x + 1)}:4.5) -- ({360/6 * \x}:4.5);
        \draw ({360/6 * (\x)}:3) -- ({360/6 * \x}:4.5);
        }
        \draw ({360/6 * (0)}:3) -- ({360/6 * 2}:3);
        \draw ({360/6 * (0)}:3) -- ({360/6 * 4}:3);
        \draw ({360/6 * (4)}:3) -- ({360/6 * 2}:3);

        \draw ({360/6 * (1)}:3) -- ({360/6 * 3}:3);
        \draw ({360/6 * (1)}:3) -- ({360/6 * 5}:3);
        \draw ({360/6 * (3)}:3) -- ({360/6 * 5}:3);

        \draw (2,1) -- ({360/6 * 5}:3);
        \draw (2,1) -- ({360/6 * 5}:4.5);
        \draw (2,1) -- ({360/6 * 2}:3);
        \draw (2,1) -- ({360/6 * 2}:4.5);
        \draw[fill] (2,1) circle (1.8pt);

        \draw (0,-2.3) -- ({360/6 * 0}:3);
        \draw (0,-2.3) -- ({360/6 * 0}:4.5);
        \draw (0,-2.3) -- ({360/6 * 3}:3);
        \draw (0,-2.3) -- ({360/6 * 3}:4.5);
        \draw[fill] (0,-2.3) circle (1.8pt);

        \draw (-2,1) -- ({360/6 * 1}:3);
        \draw (-2,1) -- ({360/6 * 1}:4.5);
        \draw (-2,1) -- ({360/6 * 4}:3);
        \draw (-2,1) -- ({360/6 * 4}:4.5);
        \draw[fill] (-2,1) circle (1.8pt);
\end{tikzpicture} \quad
\begin{tikzpicture}[scale=0.65]
         \foreach \x in {0,1,...,4}{
        \draw[fill] (\x*360/5+18:1.1) circle (1.8pt);
        \draw[fill] (\x*360/5+18:2.4) circle (1.8pt);
        \draw[fill] (\x*360/5+18+180:4.5) circle (1.8pt);

        \draw ({360/5 * (\x + 1)+18+180}:4.5) -- ({360/5 * (\x )+18+180}:4.5);
        \draw ({360/5 * (\x + 2)+18}:1.1) -- ({360/5 * (\x )+18}:1.1);
        }
        \draw ({360/5 * (1)+18}:2.4) -- ({360/5 * (0)+18}:1.1);
        \draw ({360/5 * (4)+18}:2.4) -- ({360/5 * (0)+18}:1.1);

        \draw ({360/5 * (1)+18}:2.4) -- ({360/5 * (2)+18}:1.1);
        \draw ({360/5 * (3)+18}:2.4) -- ({360/5 * (2)+18}:1.1);

        \draw ({360/5 * (0)+18}:2.4) -- ({360/5 * (1)+18}:1.1);
        \draw ({360/5 * (2)+18}:2.4) -- ({360/5 * (1)+18}:1.1);

        \draw ({360/5 * (2)+18}:2.4) -- ({360/5 * (3)+18}:1.1);
        \draw ({360/5 * (4)+18}:2.4) -- ({360/5 * (3)+18}:1.1);

        \draw ({360/5 * (0)+18}:2.4) -- ({360/5 * (4)+18}:1.1);
        \draw ({360/5 * (3)+18}:2.4) -- ({360/5 * (4)+18}:1.1);

        \draw ({360/5 * (3)+18+180}:4.5) -- ({360/5 * (1)+18}:2.4);
        \draw ({360/5 * (4)+18+180}:4.5) -- ({360/5 * (1)+18}:2.4);

        \draw ({360/5 * (4)+18+180}:4.5) -- ({360/5 * (2)+18}:2.4);
        \draw ({360/5 * (0)+18+180}:4.5) -- ({360/5 * (2)+18}:2.4);

        \draw ({360/5 * (0)+18+180}:4.5) -- ({360/5 * (3)+18}:2.4);
        \draw ({360/5 * (1)+18+180}:4.5) -- ({360/5 * (3)+18}:2.4);

        \draw ({360/5 * (1)+18+180}:4.5) -- ({360/5 * (4)+18}:2.4);
        \draw ({360/5 * (2)+18+180}:4.5) -- ({360/5 * (4)+18}:2.4);

        \draw ({360/5 * (2)+18+180}:4.5) -- ({360/5 * (0)+18}:2.4);
        \draw ({360/5 * (3)+18+180}:4.5) -- ({360/5 * (0)+18}:2.4);
\end{tikzpicture}
\end{center}
\caption{The two $(4,3,\underline{4})$-cages on 15 vertices. The graph on the right is the line graph of the Petersen graph.}\label{fig:43No4}
\end{figure}

\subsection{Upper Bounds on $n(k,g,\underline{g+1})$ Based on Other Constructions}
For the $(k,g)$-cage problem, several record $(k,g)$-graphs are obtained by a \textit{regular lift of a voltage graph}
\cite{ExooR1}, a concept which we now recall. For a graph $\Gamma$ (where multiple edges and loops are allowed), $D(\Gamma)$ denotes the set of
all darts of $\Gamma$ obtained by replacing each edge and loop of $\Gamma$ by a pair of opposing oriented darts $e$ and $e^{-1}$. For any finite group $G$, a \textit{voltage graph} is the pair $(\Gamma,\alpha)$, where $\Gamma$ is a graph and $\alpha$ is a \textit{voltage assignment}: a function that maps every dart $e \in D(\Gamma)$ to a group element in $G$ subject only to the property $\alpha(e^{-1})=\alpha(e)^{-1}$. The \textit{regular lift} (or sometimes called the
\textit{derived graph}) $\Gamma^{\alpha}=(V^{\alpha},E^{\alpha})$ of the voltage graph $(\Gamma,\alpha)$ has as vertex set $V^{\alpha}=V(\Gamma) \times G$ and for each dart $e=(u,v) \in D(\Gamma)$ and for each group element $h \in G$, there is an edge in $E^{\alpha}$ between $(u,h)$ and $(v,h \cdot \alpha(e))$. The starting graph $\Gamma$ is referred to as the \textit{base graph} of the construction. We remark that by definition any regular voltage graph lift of a $k$-regular graph $\Gamma$ is again $k$-regular. Also note that with this terminology, the canonical double cover of a graph $\Gamma$ is the regular lift of the voltage graph $(\Gamma,\alpha)$, where $\alpha$ maps every dart in $D(\Gamma)$ to $1$ in the additive group $\mathbb{Z}_2$.

In our paper, we restrict ourselves to investigating one specific base graph.
Let $K_{1,3}^{\text{loop}}$ be a cubic graph on 4 vertices obtained by adding a loop to each leaf of the tree $K_{1,3}.$ We discovered that the unique $(3,7,\underline{8})$-cage and the unique $(3,9,\underline{10})$-cage are both
regular lifts of the base graph $K_{1,3}^{\text{loop}}$.
The corresponding groups are $(\mathbb{Z}_9,+_9)$ and $(\mathbb{Z}_{19},+_{19})$, respectively, and the corresponding voltage assignments $\alpha$ and $\alpha'$ are as shown in Fig.~\ref{fig:37No8And39No10}. 

\begin{figure}[h!]
\begin{center}

\begin{tikzpicture}[scale=0.65]

    \foreach \x in {0,1,2}{
        \draw[fill] (\x*360/3-30:2.4) circle (1.8pt);
        \draw (\x*360/3-30:2.4) -- (0,0);
    }
    \draw[fill] (0,0) circle (1.8pt); 

    \foreach \x in {0,1,2} {
        \draw (\x*360/3-30:2.4) to[in=375+\x*120, out=285+\x*120, looseness=60] (\x*360/3-32:2.4);
    }
    \node at (0.3,1.5) {0};
    \node at (1.0,-0.2) {0};
    \node at (-1.0,-0.2) {0};
    \node at (0.7,2.9) {1};
    \node at (2.5,-0.8) {2};
    \node at (-2.5,-0.8) {4};
    
\end{tikzpicture} \quad
\begin{tikzpicture}[scale=0.65]

    \foreach \x in {0,1,2}{
        \draw[fill] (\x*360/3-30:2.4) circle (1.8pt);
        \draw (\x*360/3-30:2.4) -- (0,0);
    }
    \draw[fill] (0,0) circle (1.8pt); 

    \foreach \x in {0,1,2} {
        \draw (\x*360/3-30:2.4) to[in=375+\x*120, out=285+\x*120, looseness=60] (\x*360/3-32:2.4);
    }
    \node at (0.3,1.5) {0};
    \node at (1.0,-0.2) {0};
    \node at (-1.0,-0.2) {0};
    \node at (0.7,2.9) {1};
    \node at (2.5,-0.8) {7};
    \node at (-2.5,-0.8) {8};
    
\end{tikzpicture}
\end{center}
\caption{The $(3,7,\underline{8})$-cage (left) and $(3,9,\underline{10})$-cage (right) are regular lifts of a voltage graph using voltage groups $(\mathbb{Z}_9,+_9)$ and $(\mathbb{Z}_{19},+_{19})$, respectively. }\label{fig:37No8And39No10}
\end{figure}
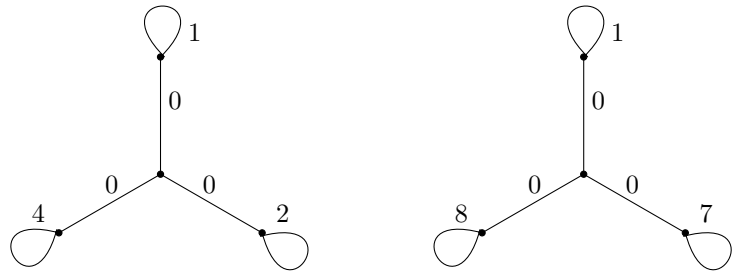

We now recall a theorem by Exoo and the second author~\cite{ExooR2}:

\begin{theorem}[Th. 3.3 in~\cite{ExooR2}]
    For every graph $\Gamma$ of girth $g$ and every integer $k>1$, there exists a regular voltage graph lift of $\Gamma$ with girth at least $kg.$
\end{theorem}

This yields the following direct corollary.

\begin{corollary}
There exist regular voltage graph lifts of $K_{1,3}^{\text{loop}}$ of arbitrarily large girth.
\end{corollary}
This motivates the definition of $n_{K_{1,3}^{\text{loop}}}(3,g,\underline{g+1})$: the order of the smallest $(3,g,\underline{g+1})$-graph which is a regular voltage graph lift of $K_{1,3}^{\text{loop}}$ (we set this value to $\infty$ if no such graph exists; we have not been able to show the universal existence of regular voltage graph lifts of $K_{1,3}^{\text{loop}}$ which are $(3,g,\underline{g+1})$-graphs for all $g \geq 3 $). 
Note that having the list of all finite groups of orders not exceeding a specific value $N$ potentially allows 
one to determine all regular voltage graph lifts of $K_{1,3}^{\text{loop}}$ of orders not exceeding $4N$ (the order
of a lift of $K_{1,3}^{\text{loop}}$ via a voltage assignment by elements of a group of order $N$). Even though this appears to be a 
formidable task, the key parameter of such a search is the number of groups in such a list, as the number
of distinct regular lifts of $K_{1,3}^{\text{loop}}$ using voltages from a fixed group $G$ of order not exceeding $N$ is bounded from above by the number of distinct $3$-element subsets of $G$, which is at most cubic in $N$.

This observation allowed us to quickly determine the exact values $n_{K_{1,3}^{\text{loop}}}(3,7,\underline{8})= 36$ and 
$n_{K_{1,3}^{\text{loop}}}(3,9,\underline{10}) = 76$ (with the first result requiring checking all groups of orders 
not exceeding $36/4$, and the second all groups of orders $ \leq 76/4 $).
Clearly $n(3,g,\underline{g+1}) \leq n_{K_{1,3}^{\text{loop}}}(3,g,\underline{g+1})$, and therefore, $36=n(3,7,\underline{8})=n_{K_{1,3}^{\text{loop}}}(3,7,\underline{8})$ and $76=n(3,9,\underline{10})=n_{K_{1,3}^{\text{loop}}}(3,9,\underline{10})$. 
The above two results appear to suggest that $n_{K_{1,3}^{\text{loop}}}(3,g,\underline{g+1})$ might also be of interest for higher girths than those for which we already determined the value $n(3,g,\underline{g+1}).$

In view of this, we computationally determined that $n_{K_{1,3}^{\text{loop}}}(3,11,\underline{12})=288$ (and therefore $n(3,11,\underline{12}) \leq 288$). We refer the interested reader to the Appendix for more details on these computations. This case is of particular interest, as the canonical double cover of the corresponding graph is a $(3,14)$-graph with 576 vertices. This is unfortunately larger than the record $(3,14)$-graph on 384 vertices~\cite{E02}; where 
$g=14$ is the second smallest girth for which the exact value of $n(3,g)$ is not known (with $g=13$ being the smallest)
\cite{ExooR1}.

Another way of obtaining upper bounds for $n(k,g,\underline{g+1})$ is by searching for $(k,g,\underline{g+1})$-graphs through large lists of `promising' graphs that have been studied in a different context. As we saw in Table~\ref{tab:overviewExactComputations}, many of the known $(k,g,\underline{g+1})$-cages have few vertex orbits (and are thus very symmetrical). We searched for $(k,g,\underline{g+1})$-graphs in the list of cubic vertex-transitive graphs up to the order 1280~\cite{PSV13}, cubic arc-transitive graphs up to the order 2048~\cite{CD02}, cubic Cayley graphs up to the order 4094~\cite{PSV13}, cubic arc-transitive and semi-symmetric graphs up to the order 10 000~\cite{CD02}, tetravalent edge-transitive graphs up to the order 512~\cite{WP20}, tetravalent arc-transitive graphs up to the order 640~\cite{PSV13,PSV15}, tetravalent 2-arc-transitive graphs up to the order 2000~\cite{P09} and pentavalent arc-transitive graphs up to the order 500~\cite{P24}, and obtained the following upper bounds\footnote{The corresponding graphs are made publicly available at:\\\url{https://github.com/JorikJooken/KGNoGPlus1Graphs}.}: 

\begin{enumerate}
\item $n(3,11,\underline{12}) \leq 272$ (graph from~\cite{PSV13})
\item $n(3,13,\underline{14}) \leq 800$ (graph from~\cite{PSV13})
\item $n(3,15,\underline{16}) \leq 2162$ (graph from~\cite{CD02})
\item $n(3,17,\underline{18}) \leq 5456$ (graph from~\cite{CD02})
\item $n(4,7,\underline{8}) \leq 273$ (graph from~\cite{PSV13,PSV15})
\item $n(4,9,\underline{10}) \leq 1518$ (graph from~\cite{P09})
\item $n(5,5,\underline{6}) \leq 192$ (graph from~\cite{P24})
\end{enumerate}

This illustrates that for larger girths the equality $n(3,g,\underline{g+1})=n_{K_{1,3}^{\text{loop}}}(3,g,\underline{g+1})$ no longer holds. Unfortunately, the canonical double covers of none of the above graphs results in a record $(k,g)$-graph.

\subsection{Proving a Conjecture by Campbell~\cite{C97}}

The~\textit{odd girth} of a graph $G$ is the length of a shortest cycle of odd length in $G$ (or $\infty$ if there is no odd cycle). In 1997, Campbell~\cite{C97} proved that the smallest $(3,6,\underline{7})$-graph with odd girth 11 has 28 vertices and conjectured that there is a unique such graph. Based on the algorithms described in the current paper, we were able to prove the uniqueness of this graph (shown in Fig.~\ref{fig:36No7OddGirth11}).

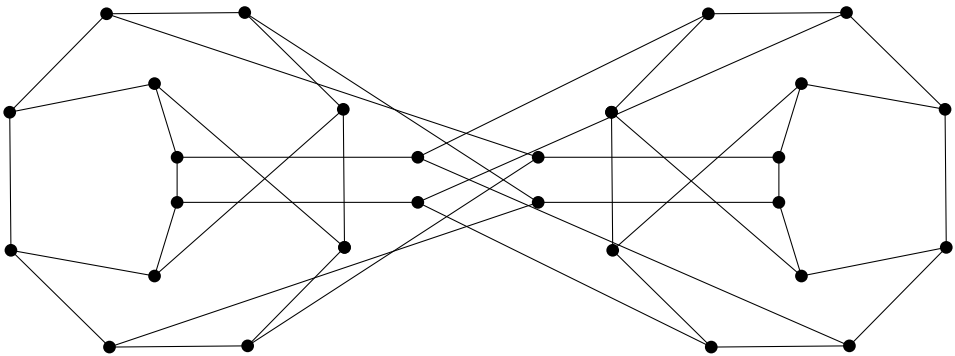
\begin{figure}[h!]
\begin{center}
\begin{tikzpicture}[scale=0.8]
  \def\sides{8}
  \def\radius{3}

\begin{scope}[rotate=90]
\begin{scope}[shift={(0,5)},rotate=180]
  \foreach \i in {1,...,\sides} {
    \draw ({360/\sides * (\i + 1)+23}:\radius) --
    ({360/\sides * (\i)+23}:\radius);
  }
  \fill ({360/\sides * 1+23}:\radius) circle (3pt);
  \fill ({360/\sides * 2+23}:\radius) circle (3pt);
  \fill ({360/\sides * 3+23}:\radius) circle (3pt);
  \fill ({360/\sides * 4+23}:\radius) circle (3pt);
  \fill ({360/\sides * 5+23}:\radius) circle (3pt);
  \fill ({360/\sides * 6+23}:\radius) circle (3pt);
  \fill ({360/\sides * 7+23}:\radius) circle (3pt);
  \fill ({360/\sides * 8+23}:\radius) circle (3pt);
  \fill ({360/\sides * 9+23}:\radius) circle (3pt);
  
  \fill (0.375,0) circle (3pt);
  \fill (-0.375,0) circle (3pt);
  \fill (1.6,-0.375) circle (3pt);
  \fill (-1.6,-0.375) circle (3pt);

  \draw (1.6,-0.375) -- ({360/\sides * 2+23}:\radius);
  \draw (-1.6,-0.375) -- ({360/\sides * 1+23}:\radius);
  \draw (0.375,0) -- (-0.375,0);
  \draw (0.375,0) -- (1.6,-0.375);
  \draw (-0.375,0) -- (-1.6,-0.375);
  \draw (1.6,-0.375) -- ({360/\sides * 6+23}:\radius);
  \draw (-1.6,-0.375) -- ({360/\sides * 5+23}:\radius);

  \draw (0.375,0) -- (0.375,1+3.0);
  \draw (-0.375,0) -- (-0.375,1+3.0);
  \draw ({360/\sides * 8+23}:\radius) -- (-0.375,1+5.0);
  \draw ({360/\sides * 3+23}:\radius) -- (0.375,1+5.0);
  \draw ({360/\sides * 4+23}:\radius) -- (-0.375,1+5.0);
  \draw ({360/\sides * 7+23}:\radius) -- (0.375,1+5.0);
\end{scope}

\begin{scope}[shift={(0,-5)},rotate=0]
  \foreach \i in {1,...,\sides} {
    \draw ({360/\sides * (\i + 1)+23}:\radius) --
    ({360/\sides * (\i)+23}:\radius);
  }
  \fill ({360/\sides * 1+23}:\radius) circle (3pt);
  \fill ({360/\sides * 2+23}:\radius) circle (3pt);
  \fill ({360/\sides * 3+23}:\radius) circle (3pt);
  \fill ({360/\sides * 4+23}:\radius) circle (3pt);
  \fill ({360/\sides * 5+23}:\radius) circle (3pt);
  \fill ({360/\sides * 6+23}:\radius) circle (3pt);
  \fill ({360/\sides * 7+23}:\radius) circle (3pt);
  \fill ({360/\sides * 8+23}:\radius) circle (3pt);
  \fill ({360/\sides * 9+23}:\radius) circle (3pt);
  
  \fill (0.375,0) circle (3pt);
  \fill (-0.375,0) circle (3pt);
  \fill (1.6,-0.375) circle (3pt);
  \fill (-1.6,-0.375) circle (3pt);

  \draw (1.6,-0.375) -- ({360/\sides * 2+23}:\radius);
  \draw (-1.6,-0.375) -- ({360/\sides * 1+23}:\radius);
  \draw (0.375,0) -- (-0.375,0);
  \draw (0.375,0) -- (1.6,-0.375);
  \draw (-0.375,0) -- (-1.6,-0.375);
  \draw (1.6,-0.375) -- ({360/\sides * 6+23}:\radius);
  \draw (-1.6,-0.375) -- ({360/\sides * 5+23}:\radius);

  \draw (0.375,0) -- (0.375,1+3.0);
  \draw (-0.375,0) -- (-0.375,1+3.0);
  \draw ({360/\sides * 8+23}:\radius) -- (0.375,1+5.0);
  \draw ({360/\sides * 3+23}:\radius) -- (-0.375,1+5.0);
  \draw ({360/\sides * 4+23}:\radius) -- (0.375,1+5.0);
  \draw ({360/\sides * 7+23}:\radius) -- (-0.375,1+5.0);
\end{scope}

\fill (0.375,1) circle (3pt);
\fill (-0.375,1) circle (3pt);
\fill (0.375,-1) circle (3pt);
\fill (-0.375,-1) circle (3pt);
\end{scope}

\end{tikzpicture}
\end{center}
\caption{The unique smallest $(3,6,\underline{7})$-graph with odd girth 11.}\label{fig:36No7OddGirth11} 
\end{figure}

\subsection{Independent Verifications}
Since the above results rely on computations, it is very important to take extra measures that could indicate a 
potential bug in the implementation of the algorithm. 
Therefore, we tried to independently verify these results using alternative existing algorithms. More specifically, we used the generators \textit{snarkhunter}~\cite{BGM11} and \textit{GENREG}~\cite{M99} that can generate $3$-regular graphs and $k$-regular graphs (with girth at least $g$), respectively. For all relevant graphs we also verified whether or not they have a $(g+1)$-cycle in two independent ways (by using a backtracking algorithm and by calculating the girth of the canonical double cover). We remark that these generation algorithms are more general purpose than our algorithm (these algorithms ignore the constraint about non-existence of $(g+1)$-cycles) and therefore also a lot slower, so of course we could only independently verify the results that we obtained regarding relatively small graphs. We were able to independently obtain the following results:
\begin{enumerate}
    \item $n(3,3,\underline{4})=10$
    \item $n(3,5,\underline{6})=18$
    \item $n(3,7,\underline{8})=36$
    \item $n(4,3,\underline{4})=15$
\end{enumerate}

Moreover, we also verified for several orders that are larger than the order of the $(k,g,\underline{g+1})$-cage that all different methods obtain the exact same graphs. As expected, all results are in agreement with each other, which gives us confidence that there are no bugs in the implementation of our algorithm.


\section*{Acknowledgements}
\noindent The computational resources and services used in this work were provided by the VSC (Flemish Supercomputer Centre), funded by the Research Foundation Flanders (FWO) and the Flemish Government - Department EWI. Jorik Jooken is supported by a Postdoctoral Fellowship of the Research Foundation Flanders (FWO) with grant number 1222524N. 
Leonard Eze and Robert Jajcay are supported by APVV grants 23-0076 and SK-AT-23-0019 and by VEGA 1/0437/23.

	\newpage	


 \newpage
 \section*{Appendix}
 \subsection*{Computational Details for Determining $n_{K_{1,3}}^{\text{loop}}(3,g,\underline{g+1})$}
The most straightforward approach for computing $n_{K_{1,3}^{\text{loop}}}(3,g,\underline{g+1})$ would be to generate all regular voltage graph lifts of $K_{1,3}^{\text{loop}}$ using increasingly large groups until a $(3,g,\underline{g+1})$-graph is encountered (note that a list of all small pairwise non-isomorphic groups is indeed available in the \textit{Small Groups Library} of \textit{GAP}~\cite{BEO02}). However, with some additional care $n_{K_{1,3}^{\text{loop}}}(3,g,\underline{g+1})$ can be computed much more efficiently without explicitly having to calculate all voltage assignments and all corresponding regular graph lifts. Note that in order to determine the voltage assignments $\alpha$ for a given group $G$ one can orient each edge (and loop) of $K_{1,3}^{\text{loop}}$ arbitrarily and then assign a group element to each such dart $e$ (this immediately fixes $\alpha(e^{-1})=\alpha(e)^{-1}$). This results in $|G|^6$ different voltage assignments for the graph $K_{1,3}^{\text{loop}}$. However, for an arbitrary connected graph $\Gamma$ and an arbitrary spanning tree $T$ of $\Gamma$, every regular voltage graph lift of $\Gamma$ is isomorphic to a regular voltage graph lift of $\Gamma$ where the voltage assignment maps every dart in $T$ to the identity element of $G$~\cite{GT01}. Therefore, we may fix the voltage assignments for the darts of the (unique) spanning tree of $K_{1,3}^{\text{loop}}$, leaving us with only (at most) $|G|^3$ different voltage assignments to consider. It is often not necessary to consider all $|G|^3$ such assignments. More specifically, for the remaining 3 darts in $K_{1,3}^{\text{loop}}$, one can use a recursive backtracking algorithm that assigns a group element to each dart and backtrack as soon as this leads to a cycle of length smaller than $g$ or length $g+1$. We also note that it is not necessary to explicitly search for cycles in the regular voltage graph lift $\Gamma^{\alpha}$, since each cycle of length $\ell$ in $\Gamma^{\alpha}$ corresponds to a closed non-reversible walk of length $\ell$ in $\Gamma$ whose net voltage is equal to the identity element of $G$ (the net voltage of a walk $W=e_1, e_2, \ldots, e_{\ell}$ is equal to $\alpha(e_1) \cdot \alpha(e_2) \cdot \ldots \cdot \alpha(e_{\ell})$). For the graph $K_{1,3}^{\text{loop}}$, it is straightforward to enumerate all closed non-reversible walks of length $\ell$. Finally, we note that if $K_{1,3}^{\text{loop}}$ has an automorphism which maps edge $e_1$ to edge $e_2$ and vice-versa, we may impose without loss of generality that $\alpha(e_1') \leq \alpha(e_2')$ (for an arbitrary but fixed ordering of the group elements, where $e_1'$ and $e_2'$ represent the darts corresponding to $e_1$ and $e_2$ in the orientation of $K_{1,3}^{\text{loop}}$, respectively). All the above considerations make it possible to compute $n_{K_{1,3}^{\text{loop}}}(3,g,\underline{g+1})$ much more efficiently than $n(3,g,\underline{g+1})$ can be computed.
\end{document}